\documentclass[a4paper]{scrartcl}

\usepackage[utf8]{inputenc}
\usepackage[T1]{fontenc}
\usepackage{lmodern}

\usepackage[svgnames,dvipsnames,rgb]{xcolor} 

\usepackage{amsfonts}
\usepackage{amsmath} 
\usepackage{amssymb}
\usepackage{amsthm}
\usepackage{graphicx}
\usepackage{hyperref}
\hypersetup{linktocpage=true,colorlinks=true,linkcolor=RoyalBlue,citecolor=PineGreen,urlcolor=RoyalBlue}
\usepackage{enumerate}
\usepackage{float}

\usepackage{tikz,pgf}
\usetikzlibrary{calc}

\newtheorem{theorem}{Theorem}
\newtheorem{lemma}[theorem]{Lemma}
\newtheorem{proposition}[theorem]{Proposition}

\theoremstyle{definition}\newtheorem{definition}[theorem]{Definition}
\theoremstyle{definition}\newtheorem{example}[theorem]{Example}
\theoremstyle{definition}

\newcommand{\Aut}{\operatorname{Aut}}
\newcommand{\Mod}{\operatorname{ mod}}

\newcommand{\sage}{\textsc{SageMath}}
\newcommand{\FlexRiLoG}{\textsc{FlexRiLoG}}

\newcommand{\blue}{\text{blue}}
\newcommand{\red}{\text{red}}

\newcommand{\conjugate}[1]{\overline{#1}}

\DeclareMathOperator{\UpairsCn}{U_{\Cn}}
\newcommand{\upairsCn}[1]{\UpairsCn(#1)}

\DeclareMathOperator{\CDCCn}{CDC_{\Cn}}
\newcommand{\cdcCn}[1]{\CDCCn(#1)}
 
\newcommand{\RR}{\mathbb{R}}
\newcommand{\CC}{\mathbb{C}}

\usepackage{etoolbox}
\ifcsdef{C}{%
\renewcommand{\C}{\mathcal C}%
}{%
\newcommand{\C}{\mathcal C}%
}

\newcommand{\Cn}{\mathcal{C}_n} 
\newcommand{\Wfun}[2]{W_{#1,#2}}
\newcommand{\Zfun}[2]{Z_{#1,#2}}

\colorlet{ecol}{black!50!white}

\definecolor{colR}{rgb}{.932,.172,.172} 
\definecolor{colB}{rgb}{.255,.41,.884} 
\colorlet{colG}{Gold}
\colorlet{col1}{LightGreen}
\colorlet{col2}{IndianRed!80!white}
\colorlet{col3}{Gold}
\colorlet{col4}{LightSkyBlue}
\colorlet{col5}{BurlyWood}

\tikzstyle{vertex}=[circle, draw, fill=black, inner sep=0pt, minimum size=4pt]
\tikzstyle{smallvertex}=[circle, draw, fill=black, inner sep=0pt, minimum size=2pt]
\tikzstyle{midvertex}=[circle, draw, fill=black, inner sep=0pt, minimum size=3pt]

\tikzstyle{lnode}=[circle,white,draw=black!60!white,fill=black!60!white,inner sep=1pt, font=\scriptsize]
\tikzstyle{lnodesmall}=[circle,white,draw=black!60!white,fill=black!60!white,inner sep=1pt, font=\scriptsize]

\colorlet{colvR}{black!70!white}
\tikzstyle{lnodeR}=[circle,colvR,draw=colvR,fill=white,inner sep=1.2pt, font=\scriptsize]
\tikzstyle{vertexR}=[circle,thick,draw=colvR,fill=white,inner sep=0pt, minimum size=4.5pt]
\tikzstyle{midvertexR}=[circle,thick,draw=colvR, fill=white, inner sep=0pt, minimum size=3.25pt]
\tikzstyle{smallvertexR}=[circle,thick,draw=colvR, fill=white, inner sep=0pt, minimum size=2pt]

\tikzstyle{extlabel}=[circle,black,draw=white,fill=white,inner sep=1pt, font=\scriptsize]

\tikzstyle{edge}=[line width=1.5pt,ecol]
\tikzstyle{redge}=[edge,colR]
\tikzstyle{bedge}=[edge,colB]
\tikzstyle{gedge}=[edge,colG]

\tikzstyle{gridl}=[ecol]
\tikzstyle{gridp}=[inner sep=1pt,circle,fill=black!70!white]
\tikzstyle{sym}=[ecol,dashed]

\tikzstyle{axes}=[gridl,-latex]

\colorlet{ncol}{Green!60!black}
\tikzstyle{nvertex}=[vertex, draw=ncol, fill=ncol]
\tikzstyle{edgeq}=[edge,gray!60,densely dashed]
\tikzstyle{nedge}=[edge,ncol]
\tikzstyle{oedge}=[edge,Red!60!black]
\tikzstyle{bdedge}=[line width=1.5pt,colB, densely dashed]
\tikzstyle{rdedge}=[line width=1.5pt,colR, densely dashed]

\title{Flexible placements of graphs with rotational symmetry}

\author{Sean Dewar\thanks{Johann Radon Institute for Computational and Applied Mathematics (RICAM), Austrian Academy of Sciences}
\and Georg Grasegger\footnotemark[1]
\and Jan Legerský\thanks{Johannes Kepler University Linz, Research Institute for Symbolic Computation (RISC)}
\thanks{Department of Applied Mathematics, Faculty of Information Technology, Czech Technical University in Prague}}
\date{}

\begin{document}

\maketitle

\begin{abstract}
We study the existence of an $n$-fold rotationally symmetric placement of a symmetric graph in the plane
allowing a continuous deformation that preserves the symmetry and the distances between adjacent vertices.
We show that such a flexible placement exists if and only if the graph has a NAC-colouring satisfying an additional property on the symmetry;
a NAC-colouring is a surjective edge colouring by two colours such that
every cycle is either monochromatic,
or there are at least two edges of each colour.    
\end{abstract}

Rigid graphs are those which have only finitely many non-congruent placements in the plane 
with the same edge lengths as a generic placement.
These graphs can, however, have non-generic special choices of a placement that can be continuously deformed by an edge length preserving motion into a non-congruent placement.
We call such a placement \emph{flexible}.
The study of flexible placements of generically rigid graphs has a long history.
Dixon found two types of flexible placements of the bipartite graph $K_{3,3}$ \cite{Dixon,Wunderlich1976,Stachel}.
Walter and Husty~\cite{WalterHusty} proved that these are indeed all (assuming that vertices do not overlap).
Figure~\ref{fig:symmetricDixon} shows some special symmetric cases of these two constructions applied to $K_{4,4}$.

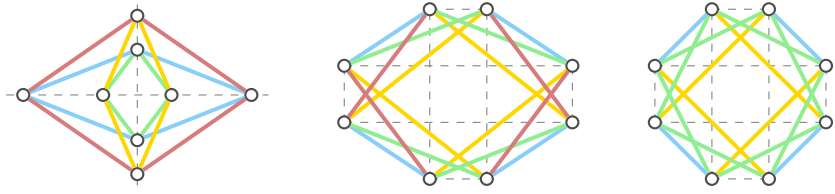
\begin{figure}[ht]
 	\centering
		\begin{tikzpicture}[scale=0.75]
			\draw[gridl, dashed] (-2.3,0)edge(2.3,0);
			\draw[gridl, dashed] (0,1.6)edge(0,-1.6);
			\node[vertexR] (2) at (2, 0) {};
			\node[vertexR] (5) at (-0.6, 0) {};
			\node[vertexR] (7) at (0.6, 0) {};
			\node[vertexR] (3) at (-2,  0) {};
			\node[vertexR] (1) at (0, -1.4) {};
			\node[vertexR] (6) at (0,  0.8) {};
			\node[vertexR] (4) at (0, -0.8) {};
			\node[vertexR] (8) at (0, 1.4) {};
			\draw[edge, col1] (6)edge(5) (5)edge(4) (7)edge(4) (7)edge(6);
			\draw[edge, col4] (3)edge(4) (3)edge(6) (2)edge(4) (2)edge(6);
			\draw[edge, col3] (8)edge(5) (7)edge(8) (1)edge(5) (7)edge(1) ;
			\draw[edge, col2] (2)edge(1) (2)edge(8) (8)edge(3) (1)edge(3);
		\end{tikzpicture}
		\qquad
		\begin{tikzpicture}[scale=0.75]
			\draw[gridl, dashed] (1.5,0)edge(2.5,0) (1.5,3)edge(2.5,3) (0,1)edge(4,1) (0,2)edge(4,2);
			\draw[gridl, dashed] (0,1)edge(0,2) (4,1)edge(4,2) (1.5,0)edge(1.5,3) (2.5,0)edge(2.5,3); 
			\node[vertexR] (7) at (1.5,3) {};
			\node[vertexR] (5) at (4,1) {};
			\node[vertexR] (3) at (2.5,0) {};
			\node[vertexR] (6) at (0,1) {};
			\node[vertexR] (2) at (1.5,0) {};
			\node[vertexR] (4) at (4,2) {};
			\node[vertexR] (1) at (0,2) {};
			\node[vertexR] (8) at (2.5,3) {};
			\draw[edge, col4] (4)edge(8) (7)edge(1) (5)edge(3) (2)edge(6) ;
			\draw[edge, col3] (6)edge(8) (3)edge(1) (2)edge(4) (5)edge(7);
			\draw[edge, col1] (1)edge(8) (2)edge(5) (7)edge(4) (3)edge(6);
			\draw[edge, col2] (5)edge(8) (2)edge(1) (7)edge(6) (3)edge(4) ;
		\end{tikzpicture}
		\qquad
		\begin{tikzpicture}[scale=0.75]
			\draw[gridl, dashed] (1,0)edge(2,0) (1,3)edge(2,3) (0,1)edge(3,1) (0,2)edge(3,2);
			\draw[gridl, dashed] (0,1)edge(0,2) (3,1)edge(3,2) (1,0)edge(1,3) (2,0)edge(2,3); 
			\node[vertexR] (7) at (1,3) {};
			\node[vertexR] (5) at (3,1) {};
			\node[vertexR] (3) at (2,0) {};
			\node[vertexR] (6) at (0,1) {};
			\node[vertexR] (2) at (1,0) {};
			\node[vertexR] (4) at (3,2) {};
			\node[vertexR] (1) at (0,2) {};
			\node[vertexR] (8) at (2,3) {};
			\draw[edge, col4] (4)edge(8) (7)edge(1) (5)edge(3) (2)edge(6) ;
			\draw[edge, col3] (6)edge(8) (3)edge(1) (2)edge(4) (5)edge(7);
			\draw[edge, col1] (5)edge(8) (2)edge(1) (7)edge(6) (3)edge(4) (1)edge(8) (2)edge(5) (7)edge(4) (3)edge(6);
		\end{tikzpicture}
	\caption{The vertices of $K_{4,4}$ can be placed symmetrically on orthogonal lines to make the graph flexible with $2$-fold rotational symmetry (left).
		A $2$-fold rotationally symmetric flexible instance of $K_{4,4}$ is obtained by placing the vertices of each part to a rectangle
		so that the two rectangles have the same intersection of diagonals and parallel/orthogonal edges (middle).
		Although there is a $4$-fold rotationally symmetric choice of rectangles (right),
		the deformed placements preserving the edge lengths are only $2$-fold symmetric. The colours indicate equality of edge lengths in a placement.}
	\label{fig:symmetricDixon}
\end{figure}

In recent works \cite{flexibleLabelings,movableGraphs,GGLSsphereflex} a deeper analysis of existence of flexible placements is done via graph colourings.
There is a special type of edge colourings, called NAC-colourings (``No Almost Cycles'', see \cite{flexibleLabelings}),
which classify the existence of a flexible placement in the plane and give a construction of the motion. 
Furthermore,
determining the NAC-colourings of a given reasonably large graph
and the corresponding constructions can be done by using the \sage{} package \FlexRiLoG{}~\cite{flexrilog}.
In \cite{GLSbridges19} we used these methods for constructing flexible placements for symmetric graphs as in Figure~\ref{fig:star}.
However, we did not take advantage of the symmetry, and instead had to construct the framework manually.

\begin{figure}[ht]
  \centering
    \begin{tikzpicture}[scale=0.45]
			\begin{scope}[xshift=0.cm,yshift=0.cm]
				\node[smallvertexR] (1) at (2.866,-0.5) {};
				\node[smallvertexR] (2) at (3.732,0.) {};
				\node[smallvertexR] (3) at (4.232,0.866) {};
				\node[smallvertexR] (4) at (4.232,1.866) {};
				\node[smallvertexR] (5) at (3.732,2.732) {};
				\node[smallvertexR] (6) at (2.866,3.232) {};
				\node[smallvertexR] (7) at (1.866,3.232) {};
				\node[smallvertexR] (8) at (1.,2.732) {};
				\node[smallvertexR] (9) at (0.5,1.866) {};
				\node[smallvertexR] (10) at (0.5,0.866) {};
				\node[smallvertexR] (11) at (1.,0.) {};
				\node[smallvertexR] (12) at (1.866,-0.5) {};
				\node[smallvertexR] (13) at (1.366,-1.366) {};
				\node[smallvertexR] (14) at (2.366,-1.366) {};
				\node[smallvertexR] (15) at (3.232,-0.866) {};
				\node[smallvertexR] (16) at (4.232,-0.866) {};
				\node[smallvertexR] (17) at (4.732,0.) {};
				\node[smallvertexR] (18) at (4.732,1.) {};
				\node[smallvertexR] (19) at (5.232,1.866) {};
				\node[smallvertexR] (20) at (4.732,2.732) {};
				\node[smallvertexR] (21) at (3.866,3.232) {};
				\node[smallvertexR] (22) at (3.366,4.098) {};
				\node[smallvertexR] (23) at (2.366,4.098) {};
				\node[smallvertexR] (24) at (1.5,3.598) {};
				\node[smallvertexR] (25) at (0.5,3.598) {};
				\node[smallvertexR] (26) at (0.,2.732) {};
				\node[smallvertexR] (27) at (0.,1.732) {};
				\node[smallvertexR] (28) at (-0.5,0.866) {};
				\node[smallvertexR] (29) at (0.,0.) {};
				\node[smallvertexR] (30) at (0.866,-0.5) {};
				\node[smallvertexR] (31) at (0.,-1.) {};
				\node[smallvertexR] (32) at (-0.5,-0.134) {};
				\node[smallvertexR] (33) at (-1.366,0.366) {};
				\node[smallvertexR] (34) at (-0.866,1.232) {};
				\node[smallvertexR] (35) at (-0.866,2.232) {};
				\node[smallvertexR] (36) at (-0.366,3.098) {};
				\node[smallvertexR] (37) at (-0.366,4.098) {};
				\node[smallvertexR] (38) at (0.634,4.098) {};
				\node[smallvertexR] (39) at (1.5,4.598) {};
				\node[smallvertexR] (40) at (2.5,4.598) {};
				\node[smallvertexR] (41) at (3.366,5.098) {};
				\node[smallvertexR] (42) at (3.866,4.232) {};
				\node[smallvertexR] (43) at (4.732,3.732) {};
				\node[smallvertexR] (44) at (5.232,2.866) {};
				\node[smallvertexR] (45) at (6.098,2.366) {};
				\node[smallvertexR] (46) at (5.598,1.5) {};
				\node[smallvertexR] (47) at (5.598,0.5) {};
				\node[smallvertexR] (48) at (5.098,-0.366) {};
				\node[smallvertexR] (49) at (5.098,-1.366) {};
				\node[smallvertexR] (50) at (4.098,-1.366) {};
				\node[smallvertexR] (51) at (3.232,-1.866) {};
				\node[smallvertexR] (52) at (2.232,-1.866) {};
				\node[smallvertexR] (53) at (1.366,-2.366) {};
				\node[smallvertexR] (54) at (0.866,-1.5) {};
				\draw[edge] (1)edge(2);
				\draw[edge] (1)edge(12);
				\draw[edge] (2)edge(3);
				\draw[edge] (3)edge(4);
				\draw[edge] (4)edge(5);
				\draw[edge] (5)edge(6);
				\draw[edge] (6)edge(7);
				\draw[edge] (7)edge(8);
				\draw[edge] (8)edge(9);
				\draw[edge] (9)edge(10);
				\draw[edge] (10)edge(11);
				\draw[edge] (11)edge(12);
				\draw[edge] (13)edge(52);
				\draw[edge] (13)edge(53);
				\draw[edge] (14)edge(51);
				\draw[edge] (15)edge(50);
				\draw[edge] (15)edge(51);
				\draw[edge] (16)edge(48);
				\draw[edge] (16)edge(49);
				\draw[edge] (17)edge(47);
				\draw[edge] (18)edge(46);
				\draw[edge] (18)edge(47);
				\draw[edge] (19)edge(44);
				\draw[edge] (19)edge(45);
				\draw[edge] (20)edge(43);
				\draw[edge] (21)edge(42);
				\draw[edge] (21)edge(43);
				\draw[edge] (22)edge(40);
				\draw[edge] (22)edge(41);
				\draw[edge] (23)edge(39);
				\draw[edge] (24)edge(38);
				\draw[edge] (24)edge(39);
				\draw[edge] (25)edge(36);
				\draw[edge] (25)edge(37);
				\draw[edge] (26)edge(35);
				\draw[edge] (27)edge(34);
				\draw[edge] (27)edge(35);
				\draw[edge] (28)edge(32);
				\draw[edge] (28)edge(33);
				\draw[edge] (29)edge(31);
				\draw[edge] (30)edge(31);
				\draw[edge] (30)edge(54);
				\draw[edge] (13)edge(14);
				\draw[edge] (13)edge(30);
				\draw[edge] (14)edge(15);
				\draw[edge] (15)edge(16);
				\draw[edge] (16)edge(17);
				\draw[edge] (17)edge(18);
				\draw[edge] (18)edge(19);
				\draw[edge] (19)edge(20);
				\draw[edge] (20)edge(21);
				\draw[edge] (21)edge(22);
				\draw[edge] (22)edge(23);
				\draw[edge] (23)edge(24);
				\draw[edge] (24)edge(25);
				\draw[edge] (25)edge(26);
				\draw[edge] (26)edge(27);
				\draw[edge] (27)edge(28);
				\draw[edge] (28)edge(29);
				\draw[edge] (29)edge(30);
				\draw[edge] (31)edge(32);
				\draw[edge] (31)edge(54);
				\draw[edge] (32)edge(33);
				\draw[edge] (33)edge(34);
				\draw[edge] (34)edge(35);
				\draw[edge] (35)edge(36);
				\draw[edge] (36)edge(37);
				\draw[edge] (37)edge(38);
				\draw[edge] (38)edge(39);
				\draw[edge] (39)edge(40);
				\draw[edge] (40)edge(41);
				\draw[edge] (41)edge(42);
				\draw[edge] (42)edge(43);
				\draw[edge] (43)edge(44);
				\draw[edge] (44)edge(45);
				\draw[edge] (45)edge(46);
				\draw[edge] (46)edge(47);
				\draw[edge] (47)edge(48);
				\draw[edge] (48)edge(49);
				\draw[edge] (49)edge(50);
				\draw[edge] (50)edge(51);
				\draw[edge] (51)edge(52);
				\draw[edge] (52)edge(53);
				\draw[edge] (53)edge(54);
				\draw[edge] (1)edge(14);
				\draw[edge] (2)edge(15);
				\draw[edge] (2)edge(16);
				\draw[edge] (2)edge(17);
				\draw[edge] (3)edge(17);
				\draw[edge] (4)edge(18);
				\draw[edge] (4)edge(20);
				\draw[edge] (5)edge(20);
				\draw[edge] (6)edge(21);
				\draw[edge] (6)edge(22);
				\draw[edge] (6)edge(23);
				\draw[edge] (7)edge(23);
				\draw[edge] (8)edge(24);
				\draw[edge] (8)edge(26);
				\draw[edge] (9)edge(26);
				\draw[edge] (10)edge(27);
				\draw[edge] (10)edge(28);
				\draw[edge] (10)edge(29);
				\draw[edge] (11)edge(29);
				\draw[edge] (12)edge(14);
				\draw[edge] (12)edge(30);
			\end{scope}
			\pgfdeclarelayer{background layer}
			\pgfsetlayers{background layer,main}
			\begin{pgfonlayer}{background layer}
			\begin{scope}[axes/.style={dashed}]
				\coordinate (o) at ($(1)!0.5!(7)$);
				\draw[axes]  (o) -- ($(o)!2.3!(4)$);
				\draw[axes] (o) -- ($(o)!2.3!(8)$);
				\draw[axes] (o) -- ($(o)!2.3!(12)$);
			\end{scope}
			\end{pgfonlayer}
		\end{tikzpicture}
  \caption{A symmetric graph having a $3$-fold rotationally symmetric flexible placement.}%
  \label{fig:star}
\end{figure}
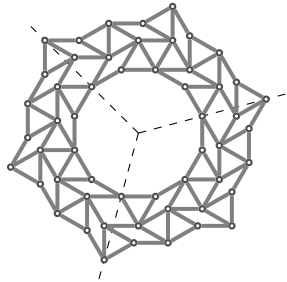

Symmetry plays an important role in art and design, and often appears in nature also.
Due to this, there is a large body of work focused on symmetric frameworks and their properties in the context of rigidity theory;
see \cite{gainsparserigid,OwenPower2012}.
In particular,
we shall be focusing on graphs and frameworks that display $n$-fold rotational symmetry, such as in Figure~\ref{fig:star}.

In this paper we formalise the NAC-colouring method for rotationally symmetric flexible placements,
in such a way that the motion preserves the symmetry.
By combining Lemma \ref{lem:flexibleImpliesNAC} and \ref{lem:NACImpliesflexible} of Section \ref{sec:nfold},
we obtain the following result:

\begin{theorem}\label{thm:main}
	A $\Cn$-symmetric connected graph has a $\Cn$-symmetric NAC-colour\-ing if and only if 
	it has a $\Cn$-symmetric flexible placement in $\RR^2$.
\end{theorem}

Similiar to \cite{movableGraphs}, 
we also identify properties of such NAC-colourings that determine when all flexible placements must have overlapping vertices.

\section{Rigid and flexible frameworks}
We briefly recall some basic notions from rigidity theory.
A \emph{framework in $\RR^2$} is a pair $(G,p)$
where $G$ is a (finite simple) graph and $p : V(G) \rightarrow \RR^2$
is a \emph{placement} of $G$, a possibly non-injective map such $p(u) \neq p(v)$ if $uv \in E(G)$. 
We define frameworks $(G,p)$ and $(G,q)$ to be \emph{equivalent} if
for all $uv \in E(G)$,
\begin{align}\label{eq:equivalence}
	\| p(u) - p(v) \| = \| q(u) - q(v)\|\,.
\end{align}
We define two placements $p,q$ of $G$ to be \emph{congruent} if \eqref{eq:equivalence} holds for all $u,v \in V(G)$;
equivalently, $p$ and $q$ are congruent if there exists a Euclidean isometry~$M$ of $\RR^2$ such that $M q(v) = p(v)$ for all $v \in V(G)$.

A \emph{flex (in $\RR^2$)} of the framework $(G,p)$ is a continuous path $t \mapsto p_t$, $t \in [0,1]$,
in the space of placements of $G$ such that $p_0= p$ and each $(G,p_t)$ is equivalent to~$(G,p)$.
If $p_t$ is congruent to $p$ for all $t \in [0,1]$ then $p_t$ is \emph{trivial}.
We define~$(G,p)$ to be \emph{flexible} if there is a non-trivial flex of $(G,p)$ in~$\RR^2$,
and \emph{rigid} otherwise.

It was shown in \cite{Geiringer1927} that a framework $(G,p)$ with a generic placement of vertices (see~\cite{GraverServatius}) is rigid if and only if
$G$ contains a Laman graph as a spanning subgraph.
This does not inform us whether a graph will have a flexible placement;
for example,
any generic placement of $K_{4,4}$ is rigid, however as shown by Figure~\ref{fig:symmetricDixon} we can construct flexible placements for it.
To determine whether a graph has flexible placements we introduce the following.

\begin{definition}
	An edge colouring $\delta: E(G) \rightarrow \{\red{}, \blue{}\}$ of a graph $G$ is
	a \emph{NAC-colouring} 
	if $\delta(E(G)) = \{ \red{},\blue{} \}$
	and for each cycle in $G$, either all edges have the same colour, or there are at least two red and two blue edges.
	NAC-colourings $\delta,\conjugate{\delta}$ of $G$ are \emph{conjugate} if $\delta(e)\neq\conjugate{\delta}(e)$ for all $e\in E(G)$.
\end{definition}
The colourings considered within this paper are not required to have incident edges coloured differently,
contrary to common graph-theoretical terminology.
Having these definitions, we can recall the result \cite[Theorem 3.1]{flexibleLabelings}.
\begin{theorem}
	A connected graph has a flexible placement in $\RR^2$ if and only if it has a NAC-colouring.
\end{theorem}

\section{Rotational symmetry}
We now recall some of the basics of rotational symmetry for graphs and frameworks.
Let $G$ be a graph and $n \geq 2$. Let the group $\Cn := \left\langle \omega : \omega^n =1 \right\rangle$
act on $G$, i.e., there exists an injective group homomorphism $\theta: \Cn \rightarrow \Aut(G)$.
We define $\gamma v := \theta(\gamma)(v)$ for $\gamma\in\Cn$;
similarly, for any edge $e=uv\in E(G)$, we define $\gamma e := \gamma u \gamma v$.
We shall define $v \in V(G)$ to be an \emph{invariant vertex} if $\gamma v =v$ for all $\gamma \in \Cn$,
and \emph{partially invariant} if $\gamma v =v$ for some $\gamma \in \Cn, \gamma\neq 1$.
The graph $G$ is called \emph{$\Cn$-symmetric} if in addition we have that every partially invariant vertex is invariant,
and the set of invariant vertices of $G$ forms an independent set.

A placement $p$ of a $\Cn$-symmetric graph $G$ in $\RR^2$ is \emph{$\Cn$-symmetric} 
if $p(\gamma v ) = \tau(\gamma)p(v)$ for each $v \in V(G)$ and $\gamma = \omega^k\in\Cn$,
where $\tau(\omega^k)$ is the $2 k \pi/n$ rotation matrix;
likewise, we define the pair $(G,p)$ to be a \emph{$\Cn$-symmetric framework}.
Note that the invariant vertices are necessarily placed at the origin, 
which explains the requirement to form an independent set. 
If there is a non-trivial flex $p_t$ of $(G,p)$ such that each $(G,p_t)$ is $\Cn$-symmetric,
then $(G,p)$ is \emph{$\Cn$-symmetric flexible} (or \emph{n-fold rotation symmetric flexible}),
and \emph{$\Cn$-symmetric rigid} otherwise.

We wish to extend our definition of NAC-colourings to the symmetric case.
Let $\delta$ be a colouring of $G$.
A \emph{\red{} component} is a connected component of 
\begin{equation*}
	G^\delta_{\red{}} := \left( V(G), \{ e \in E(G) : \delta(e) = \red{} \} \right)\,.
\end{equation*}
A \red{} component $H \subset G$ is \emph{partially invariant} if there exists $\gamma \in \Cn \setminus \{1\}$ such that $\gamma H = H$,
and \emph{invariant} if $\gamma H = H$ for all $\gamma \in \Cn$ (see Figure~\ref{fig:invariant} for an example).
We define $G^\delta_{\blue{}}$ and \blue{} (partially invariant) components analogously.

\begin{figure}[ht]
	\centering
	  \begin{tikzpicture}[scale=1]
	    \begin{scope}
	      \node[vertex] (1) at (0:1) {};
				\node[vertex] (2) at (60:1) {};
				\node[vertex] (3) at (120:1) {};
				\node[vertex] (4) at (180:1) {};
				\node[vertex] (5) at (240:1) {};
				\node[vertex] (6) at (300:1) {};
	      \draw[edge] (1)edge(2) (2)edge(3) (3)edge(4) (4)edge(5) (5)edge(6) (6)edge(1);
	      \draw[edge,colR] (1)edge(3) (3)edge(5) (5)edge(1);
	      \draw[edge] (2)edge(4) (4)edge(6) (6)edge(2);
	    \end{scope}
	    \begin{scope}[xshift=4cm]
	      \node[vertex] (1) at (0:1) {};
				\node[vertex] (2) at (60:1) {};
				\node[vertex] (3) at (120:1) {};
				\node[vertex] (4) at (180:1) {};
				\node[vertex] (5) at (240:1) {};
				\node[vertex] (6) at (300:1) {};
	      \draw[edge,colR] (1)edge(2) (2)edge(3) (3)edge(4) (4)edge(5) (5)edge(6) (6)edge(1);
	      \draw[edge] (1)edge(3) (3)edge(5) (5)edge(1);
	      \draw[edge] (2)edge(4) (4)edge(6) (6)edge(2);
	    \end{scope}
	  \end{tikzpicture}
  \caption{A partially invariant (but not invariant) red component on the left and an invariant red component (and therefore also partially invariant) on the right for $\C_6$-symmetry.
  The symmetry is indicated by the graph layout.}
  \label{fig:invariant}
\end{figure}
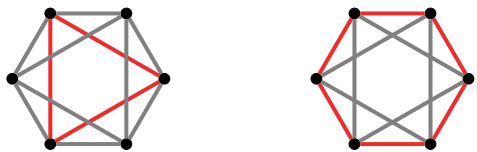

We focus on the class of NAC-colourings suitable for dealing with symmetries.

\begin{definition}\label{defn:NACnfold}
	We define a NAC-colouring $\delta$ of a $\Cn$-symmetric graph $G$ to be \emph{$\Cn$-symmetric} if
	$\delta(\gamma e) = \delta(e)$ for all $e \in E(G)$ and $\gamma \in \Cn$ and
	no two distinct \blue{}, resp.\ \red{}, partially invariant components are connected by an edge.
\end{definition}

\section{Proof of Theorem \ref{thm:main}}\label{sec:nfold}

We now prove our two key lemmas for Theorem \ref{thm:main}.
The first result (Lemma~\ref{lem:flexibleImpliesNAC}) uses tools from valuation theory;
we refer the reader to \cite{Deuring} for background reading.

Let $(G,p)$ be a $\Cn$-symmetric framework in $\RR^2$ and fix a non-invariant vertex~$v_0$ that does not lie at the origin.
By rotating $(G,p)$ we may assume that $p(v_0)$ lies on the line $\{(t,0) : t \in \RR \}$.
Define for each edge $vw$ the polynomials
\begin{align*}
	\Wfun{u}{v} := (x_u - x_v) + i (y_u - y_v), \qquad
	\Zfun{u}{v} := (x_u - x_v) - i (y_u - y_v).
\end{align*}
Using these we define the algebraic set $\mathcal{V}_n(G,p)$ in $(\CC^2)^{V(G)}$ by the equations
\begin{eqnarray*}
	y_{v_0} &=& 0, \\
	x_{\omega v} &=& \cos (2\pi/n) x_v + \sin (2 \pi/n) y_v \text{ for all } v \in V(G), \\ 
	y_{\omega v} &=& -\sin (2\pi/n) x_v + \cos (2 \pi/n) y_v \text{ for all } v \in V(G), \\
	\Wfun{u}{v} \Zfun{u}{v} &=& (x_u - x_v)^2 + (y_u - y_v)^2 = \|p_u - p_v \|^2 \text{ for all } uv \in E(G).
\end{eqnarray*}
If $G$ is connected and $(G,p)$ is $\Cn$-symmetric flexible then the algebraic set $\mathcal{V}_n(G,p)$ cannot be finite and must contain an irreducible algebraic curve $C$.
Further,
as fixing the vertex $v_0$ to the line $\{(t,0) : t \in \RR \}$ will remove any trivial flexes,
the ``angle'' between two edges must vary,
i.e., for some edges $u'v',u'w'$,
both $\Wfun{u'}{v'}/\Wfun{u'}{w'}$ and $\Zfun{u'}{v'}/\Zfun{u'}{w'}$ take an infinite amount of values over $C$.
We shall define the irreducible algebraic curve $C$ to be a \emph{$\Cn$-symmetric motion}.

\begin{lemma}\label{lem:flexibleImpliesNAC}
	If $(G,p)$ is a $\Cn$-symmetric flexible framework in $\RR^2$,  $G$ being a connected graph,
	then $G$ has a $\Cn$-symmetric NAC-colouring.
\end{lemma}

\begin{proof}
	The assumption implies that there exists a $\Cn$-symmetric motion $C$ with infinitely many real points.
	Let $u'v'$ and $u'w'$ be edges of $G$ such that the function $\Wfun{u'}{v'}/\Wfun{u'}{w'}$ takes infinitely many values.
	Then $\Wfun{u'}{v'}/\Wfun{u'}{w'}$ is a transcendental element of the function field $\CC(C)$.
	By Chevalley's Theorem (see for instance~\cite{Deuring}), there exists a valuation~$\nu$ of $\CC(C)$
	such that $0<\nu\left(\Wfun{u'}{v'}/\Wfun{u'}{w'}\right)$ and $\nu(\CC)=\{0\}$.
	Hence, we have $\alpha:=\nu(\Wfun{u'}{w'}) < \nu(\Wfun{u'}{v'})$.

	Define the colouring $\delta :E(G) \rightarrow \{ \red{}, \blue{} \}$ where for each $e =uv \in E(G)$,
	\begin{align}	\label{eq:valuation2NAC}
		\delta(e) :=
		\begin{cases}
			\red{}  &\text{ if } \nu(\Wfun{u}{v}) > \alpha \\
			\blue{} &\text{ if } \nu(\Wfun{u}{v}) \leq \alpha.
		\end{cases}
	\end{align}
	By \cite[Theorem 2.8]{movableGraphs}, we see that $\delta$ is a NAC-colouring.
	Moreover, $\delta( \omega^k e) = \delta (e)$ for all $e=uv \in E(G)$,
	since
	\begin{align*}
		\nu \left( \Wfun{ \omega^k u}{\omega^k v} \right) 
		= \nu \left(e^{\frac{-2\pi k i}{n}} \Wfun{u}{v} \right) 
		= \nu \left(e^{\frac{-2\pi k i}{n}}\right) + \nu \left( \Wfun{u}{v} \right)
		= \nu \left( \Wfun{u}{v} \right).
	\end{align*}

	It suffices to show that two distinct partially invariant \red{} components
	cannot be connected by an edge,
	since there exists a valuation $\conjugate{\nu}$ and threshold $\conjugate{\alpha}$
	yielding a NAC-colouring $\conjugate{\delta}$ conjugate to $\delta$ by \cite[Lemma~2.13]{movableGraphs}.
	Suppose for contradiction that $H_1$ and $H_2$ are distinct
	partially invariant \red{} components of $G$ that are connected
	by an edge $v_1 v_2$ with $v_j \in V(H_j)$.
	As each $H_j$ is a \red{} component and they are distinct, $\delta(v_1 v_2) = \blue{}$.
	Suppose that $v_1$ and $v_2$ are not invariant vertices.
	Due to partial invariance there exist $\omega^{k_j} v_j$ in $H_j$ for some $1 \leq k_j <n$.
	Let $(u_1, u_2 \ldots, u_m)$ be a path in $H_1$ with $u_1 = v_1$ and $u_m = \omega^{k_1} v_1$, then
	\begin{align*}
		\sum_{s = 1}^{m-1} \Wfun{u_s}{u_{s+1}} = (1 - e^{\frac{-2\pi k_1 i}{n}}) \left( x_{v_1} + i y_{v_1} \right).
	\end{align*}
	If we define $W_{v_1} :=  x_{v_1} + i y_{v_1} \in \CC(C)$ then as $(1 - e^{\frac{-2\pi k_1 i}{n}}) \in \CC$,
	\begin{align*}
		\nu (W_{v_1}) = \nu\left(\sum_{s = 1}^{m-1} \Wfun{u_j}{u_{j+1}}\right) \geq \min_{s=1, \ldots,m-1} \nu \left(\Wfun{u_s}{u_{s+1}} \right) > \alpha.
	\end{align*}
	By a similar method, if we define $W_{v_2} :=  x_{v_2} + i y_{v_2} \in \CC(C)$ then $\nu (W_{v_2}) >\alpha$.
	The vertices $v_1$ and $v_2$ cannot be both invariant, since by definition,
	invariant vertices form an independent set.
	If $v_j$ is invariant, then it must be at the origin and hence $W_{v_j}=0$.
	In any case, the edge $v_1 v_2$ cannot be \blue{} as either
		$\nu (\Wfun{v_1}{v_2}) = \nu\left( W_{v_1} - W_{v_2} \right)
		\geq \min\{ \nu (W_{v_1}), \nu (W_{v_2})\}  > \alpha$
		if $v_1,v_2$ are not invariant,
		or $\nu (\Wfun{v_1}{v_2}) = \nu\left( W_{v_j} \right) > \alpha$ if $v_j \in \{v_1,v_2\}$ is not invariant.
\end{proof}

\begin{lemma}\label{lem:NACImpliesflexible}
	If a $\Cn$-symmetric connected graph $G$ has a $\Cn$-symmetric NAC-colouring~$\delta$,
	then there exists a $\Cn$-symmetric flexible framework $(G,p)$ in $\RR^2$.  
\end{lemma}

\begin{proof}
	The proof is based on the ``zigzag'' grid construction from \cite{flexibleLabelings} with a specific choice of the grid.	
	Let $R^0_1, \dots, R^{n-1}_1, \dots, R^0_m, \dots, R^{n-1}_m$ be the \red{} components of $G_\red^\delta$
	that are not partially invariant. We can assume that $R^i_j = \omega^i R^0_j$ for $0\leq i < n$ and $1\leq j\leq m$.
	Similarly, let $B^0_1, \dots, B^{n-1}_1, \dots, B^0_k, \dots, B^{n-1}_k$
	be the \blue{} components of $G_\blue^\delta$
	that are not partially invariant and $B^i_j = \omega^i B^0_j$ for $0\leq i < n$ and $1\leq j\leq k$.

	Let $a_1, \dots, a_m$ and $b_1,\dots, b_k$ be points in $\RR^2\setminus\{(0,0)\}$ 
	such that $a_j \neq \tau(\omega)^i a_{j'}$,
	$b_j \neq \tau(\omega)^i b_{j'}$,
	and $a_j, \tau(\omega)^ib_{j'}$ are linearly independent for $j\neq j'$ and $0\leq i<n$ arbitrary.
	We define functions $\overline{a},\overline{b} \colon V(G)\rightarrow \RR^2$ by
	\begin{equation*}
		\overline{a}(v) = \begin{cases}
			\tau(\omega)^i a_j &\text{if } v \in R^i_j \\
			(0,0) 		 &\text{otherwise,}
		\end{cases}
	\quad\text{ and }\quad
	 \overline{b}(v) = \begin{cases}
			\tau(\omega)^i b_j &\text{if } v \in B^i_j \\
			(0,0) 		 &\text{otherwise.}
		\end{cases}
	\end{equation*}
	We note that a vertex is mapped to the origin by $\overline{a}$
	(respectively $\overline{b}$)
	if and only if it lies in a \red{}
	(respectively, \blue{})
	partially invariant component.
	We now obtain for each $t \in [0,2\pi]$ a placement $p_t$ of $G$ with
	\begin{equation}\label{eq:rotation}
		p_t(v) := 
		R(t) \overline{a}(v) + \overline{b}(v) \, ,
	\end{equation}
	where $R(t)$ is the rotation matrix by $t$ radians.
		
	First, we have to show that no two adjacent vertices are mapped to the same point by the placement $p_0$.
	Assume that $p_0(u) = p_0(v)$ for some vertices $u,v$.
	By construction we have $\left(\overline{a}(u),\overline{b}(u)\right)= \left(\overline{a}(v), \overline{b}(v)\right)$.	
	Suppose this is due to the fact that
	$u$ and $v$ belong to the same \red{} and same \blue{} (possibly partially invariant) component.
	Hence, $uv\notin E(G)$, otherwise $\delta$ is not a NAC-colouring ($uv$ would yield a cycle with a single edge in one colour).
	On the other hand, if $u$ and $v$ are in two different \red{} (resp.\ \blue{}) components,
	then $\overline{a}(u)=\overline{a}(v) = (0,0)$ (resp.\ $\overline{b}(u)=\overline{b}(v)=(0,0)$).
	By our construction of $\overline{a}$
	(resp.~$\overline{b}$),
	it follows that $u,v$ both lie in partially invariant \red{} (resp.\ \blue{}) components. 
	Since these components are partially invariant,
	$uv\notin E(G)$ by the assumption that $\delta$ is $\Cn$-symmetric.
	
	Now choose $uv\in E(G)$. 
	If $\delta(uv)$ is \red{}
	(resp. \blue{}) 
	then $\overline{a}(u)= \overline{a}(v)$ 
	(resp. ${\overline{b}(u)= \overline{b}(v)}$).
	Hence, the edge length $\left\| p_t(u) - p_t(v) \right\|$ is independent of $t$.
	As no two vertices connected by an edge are mapped to the same point,
	$(G,p) := (G,p_0)$ is a framework with a flex $p_t$.
	Further,
	the flex is not trivial by surjectivity of $\delta$,
	thus $(G,p)$ is a flexible framework.
	
	Finally, we show that $p_t$ is $\Cn$-symmetric.
	If $v\in R^i_j \cap B^k_\ell$,
	then 
	\begin{align*}
		\omega v\in \tau(\omega) R^i_j \cap \tau(\omega) B^k_\ell= R^{(i+1 \Mod n)}_j \cap B^{(k+1 \Mod n)}_\ell.
	\end{align*}
	Hence, $\overline{a}(\omega v) = \tau(\omega)\overline{a}(v)$ and $\overline{b}(\omega v) = \tau(\omega)\overline{b}(v)$.
	The same equalities hold also if $v$ belongs to a partially invariant component,
	since then $\omega v$ is also in a partially invariant component.
	Using \eqref{eq:rotation} and commutativity of rotation matrices, we get
	\begin{align*}
		p_t(\omega v) &=
		 \tau(\omega) 
		R(t)\overline{a}(v) +\tau( \omega) \overline{b}(v)
		 = \tau(\omega) p_t(v).
	\end{align*}
\end{proof}

\begin{example}	
	By using the construction described in Lemma \ref{lem:NACImpliesflexible} we can construct the $\Cn$-symmetric flexible frameworks given in
	Figure~\ref{fig:flexrot-574214559646161505}.
	
	\begin{figure}[ht]
	  \centering
	    \begin{tikzpicture}[scale=0.45] 
				\begin{scope}[xshift=0.cm,yshift=0.cm]
					\node[midvertexR] (1) at (-2.,0.) {};
					\node[midvertexR] (2) at (0.,0.) {};
					\node[midvertexR] (3) at (0.,2.) {};
					\node[midvertexR] (4) at (0.,-2.) {};
					\node[midvertexR] (5) at (1.,-1.) {};
					\node[midvertexR] (6) at (1.,1.) {};
					\node[midvertexR] (7) at (2.,0.) {};
					\node[midvertexR] (8) at (0.,0.) {};
					\node[midvertexR] (9) at (-1.,-1.) {};
					\node[midvertexR] (10) at (0.,0.) {};
					\node[midvertexR] (11) at (0.,0.) {};
					\node[midvertexR] (12) at (-1.,1.) {};
					\draw[bedge] (1)edge(9);
					\draw[bedge] (1)edge(10);
					\draw[bedge] (2)edge(4);
					\draw[bedge] (2)edge(5);
					\draw[bedge] (3)edge(8);
					\draw[bedge] (3)edge(12);
					\draw[bedge] (4)edge(5);
					\draw[bedge] (6)edge(7);
					\draw[bedge] (6)edge(11);
					\draw[bedge] (7)edge(11);
					\draw[bedge] (8)edge(12);
					\draw[bedge] (9)edge(10);
					\draw[redge] (1)edge(11);
					\draw[redge] (1)edge(12);
					\draw[redge] (2)edge(3);
					\draw[redge] (2)edge(6);
					\draw[redge] (3)edge(6);
					\draw[redge] (4)edge(8);
					\draw[redge] (4)edge(9);
					\draw[redge] (5)edge(7);
					\draw[redge] (5)edge(10);
					\draw[redge] (7)edge(10);
					\draw[redge] (8)edge(9);
					\draw[redge] (11)edge(12);
				\end{scope}
				\begin{scope}[xshift=4.5cm,yshift=0.cm]
					\node[midvertexR] (1) at (-1.623,0.782) {};
					\node[midvertexR] (2) at (0.782,-0.377) {};
					\node[midvertexR] (3) at (0.782,1.623) {};
					\node[midvertexR] (4) at (-0.782,-1.623) {};
					\node[midvertexR] (5) at (0.623,-1.782) {};
					\node[midvertexR] (6) at (1.782,0.623) {};
					\node[midvertexR] (7) at (1.623,-0.782) {};
					\node[midvertexR] (8) at (-0.782,0.377) {};
					\node[midvertexR] (9) at (-1.782,-0.623) {};
					\node[midvertexR] (10) at (-0.377,-0.782) {};
					\node[midvertexR] (11) at (0.377,0.782) {};
					\node[midvertexR] (12) at (-0.623,1.782) {};
					\draw[bedge] (1)edge(9);
					\draw[bedge] (1)edge(10);
					\draw[bedge] (2)edge(4);
					\draw[bedge] (2)edge(5);
					\draw[bedge] (3)edge(8);
					\draw[bedge] (3)edge(12);
					\draw[bedge] (4)edge(5);
					\draw[bedge] (6)edge(7);
					\draw[bedge] (6)edge(11);
					\draw[bedge] (7)edge(11);
					\draw[bedge] (8)edge(12);
					\draw[bedge] (9)edge(10);
					\draw[redge] (1)edge(11);
					\draw[redge] (1)edge(12);
					\draw[redge] (2)edge(3);
					\draw[redge] (2)edge(6);
					\draw[redge] (3)edge(6);
					\draw[redge] (4)edge(8);
					\draw[redge] (4)edge(9);
					\draw[redge] (5)edge(7);
					\draw[redge] (5)edge(10);
					\draw[redge] (7)edge(10);
					\draw[redge] (8)edge(9);
					\draw[redge] (11)edge(12);
				\end{scope}
				\begin{scope}[xshift=9.cm,yshift=0.cm]
					\node[midvertexR] (1) at (-0.777,0.975) {};
					\node[midvertexR] (2) at (0.975,-1.223) {};
					\node[midvertexR] (3) at (0.975,0.777) {};
					\node[midvertexR] (4) at (-0.975,-0.777) {};
					\node[midvertexR] (5) at (-0.223,-1.975) {};
					\node[midvertexR] (6) at (1.975,-0.223) {};
					\node[midvertexR] (7) at (0.777,-0.975) {};
					\node[midvertexR] (8) at (-0.975,1.223) {};
					\node[midvertexR] (9) at (-1.975,0.223) {};
					\node[midvertexR] (10) at (-1.223,-0.975) {};
					\node[midvertexR] (11) at (1.223,0.975) {};
					\node[midvertexR] (12) at (0.223,1.975) {};
					\draw[bedge] (1)edge(9);
					\draw[bedge] (1)edge(10);
					\draw[bedge] (2)edge(4);
					\draw[bedge] (2)edge(5);
					\draw[bedge] (3)edge(8);
					\draw[bedge] (3)edge(12);
					\draw[bedge] (4)edge(5);
					\draw[bedge] (6)edge(7);
					\draw[bedge] (6)edge(11);
					\draw[bedge] (7)edge(11);
					\draw[bedge] (8)edge(12);
					\draw[bedge] (9)edge(10);
					\draw[redge] (1)edge(11);
					\draw[redge] (1)edge(12);
					\draw[redge] (2)edge(3);
					\draw[redge] (2)edge(6);
					\draw[redge] (3)edge(6);
					\draw[redge] (4)edge(8);
					\draw[redge] (4)edge(9);
					\draw[redge] (5)edge(7);
					\draw[redge] (5)edge(10);
					\draw[redge] (7)edge(10);
					\draw[redge] (8)edge(9);
					\draw[redge] (11)edge(12);
				\end{scope}
				\begin{scope}[xshift=13.5cm,yshift=0.cm]
					\node[midvertexR] (1) at (-0.099,0.434) {};
					\node[midvertexR] (2) at (0.434,-1.901) {};
					\node[midvertexR] (3) at (0.434,0.099) {};
					\node[midvertexR] (4) at (-0.434,-0.099) {};
					\node[midvertexR] (5) at (-0.901,-1.434) {};
					\node[midvertexR] (6) at (1.434,-0.901) {};
					\node[midvertexR] (7) at (0.099,-0.434) {};
					\node[midvertexR] (8) at (-0.434,1.901) {};
					\node[midvertexR] (9) at (-1.434,0.901) {};
					\node[midvertexR] (10) at (-1.901,-0.434) {};
					\node[midvertexR] (11) at (1.901,0.434) {};
					\node[midvertexR] (12) at (0.901,1.434) {};
					\draw[bedge] (1)edge(9);
					\draw[bedge] (1)edge(10);
					\draw[bedge] (2)edge(4);
					\draw[bedge] (2)edge(5);
					\draw[bedge] (3)edge(8);
					\draw[bedge] (3)edge(12);
					\draw[bedge] (4)edge(5);
					\draw[bedge] (6)edge(7);
					\draw[bedge] (6)edge(11);
					\draw[bedge] (7)edge(11);
					\draw[bedge] (8)edge(12);
					\draw[bedge] (9)edge(10);
					\draw[redge] (1)edge(11);
					\draw[redge] (1)edge(12);
					\draw[redge] (2)edge(3);
					\draw[redge] (2)edge(6);
					\draw[redge] (3)edge(6);
					\draw[redge] (4)edge(8);
					\draw[redge] (4)edge(9);
					\draw[redge] (5)edge(7);
					\draw[redge] (5)edge(10);
					\draw[redge] (7)edge(10);
					\draw[redge] (8)edge(9);
					\draw[redge] (11)edge(12);
				\end{scope}
				\begin{scope}[xshift=18.cm,yshift=0.cm]
					\node[midvertexR] (1) at (-0.099,-0.434) {};
					\node[midvertexR] (2) at (-0.434,-1.901) {};
					\node[midvertexR] (3) at (-0.434,0.099) {};
					\node[midvertexR] (4) at (0.434,-0.099) {};
					\node[midvertexR] (5) at (-0.901,-0.566) {};
					\node[midvertexR] (6) at (0.566,-0.901) {};
					\node[midvertexR] (7) at (0.099,0.434) {};
					\node[midvertexR] (8) at (0.434,1.901) {};
					\node[midvertexR] (9) at (-0.566,0.901) {};
					\node[midvertexR] (10) at (-1.901,0.434) {};
					\node[midvertexR] (11) at (1.901,-0.434) {};
					\node[midvertexR] (12) at (0.901,0.566) {};
					\draw[bedge] (1)edge(9);
					\draw[bedge] (1)edge(10);
					\draw[bedge] (2)edge(4);
					\draw[bedge] (2)edge(5);
					\draw[bedge] (3)edge(8);
					\draw[bedge] (3)edge(12);
					\draw[bedge] (4)edge(5);
					\draw[bedge] (6)edge(7);
					\draw[bedge] (6)edge(11);
					\draw[bedge] (7)edge(11);
					\draw[bedge] (8)edge(12);
					\draw[bedge] (9)edge(10);
					\draw[redge] (1)edge(11);
					\draw[redge] (1)edge(12);
					\draw[redge] (2)edge(3);
					\draw[redge] (2)edge(6);
					\draw[redge] (3)edge(6);
					\draw[redge] (4)edge(8);
					\draw[redge] (4)edge(9);
					\draw[redge] (5)edge(7);
					\draw[redge] (5)edge(10);
					\draw[redge] (7)edge(10);
					\draw[redge] (8)edge(9);
					\draw[redge] (11)edge(12);
				\end{scope}
				\begin{scope}[xshift=22.5cm,yshift=0.cm]
					\node[midvertexR] (1) at (-0.777,-0.975) {};
					\node[midvertexR] (2) at (-0.975,-1.223) {};
					\node[midvertexR] (3) at (-0.975,0.777) {};
					\node[midvertexR] (4) at (0.975,-0.777) {};
					\node[midvertexR] (5) at (-0.223,-0.025) {};
					\node[midvertexR] (6) at (0.025,-0.223) {};
					\node[midvertexR] (7) at (0.777,0.975) {};
					\node[midvertexR] (8) at (0.975,1.223) {};
					\node[midvertexR] (9) at (-0.025,0.223) {};
					\node[midvertexR] (10) at (-1.223,0.975) {};
					\node[midvertexR] (11) at (1.223,-0.975) {};
					\node[midvertexR] (12) at (0.223,0.025) {};
					\draw[bedge] (1)edge(9);
					\draw[bedge] (1)edge(10);
					\draw[bedge] (2)edge(4);
					\draw[bedge] (2)edge(5);
					\draw[bedge] (3)edge(8);
					\draw[bedge] (3)edge(12);
					\draw[bedge] (4)edge(5);
					\draw[bedge] (6)edge(7);
					\draw[bedge] (6)edge(11);
					\draw[bedge] (7)edge(11);
					\draw[bedge] (8)edge(12);
					\draw[bedge] (9)edge(10);
					\draw[redge] (1)edge(11);
					\draw[redge] (1)edge(12);
					\draw[redge] (2)edge(3);
					\draw[redge] (2)edge(6);
					\draw[redge] (3)edge(6);
					\draw[redge] (4)edge(8);
					\draw[redge] (4)edge(9);
					\draw[redge] (5)edge(7);
					\draw[redge] (5)edge(10);
					\draw[redge] (7)edge(10);
					\draw[redge] (8)edge(9);
					\draw[redge] (11)edge(12);
				\end{scope}
				\begin{scope}[xshift=27.cm,yshift=0.cm]
					\node[midvertexR] (1) at (-1.623,-0.782) {};
					\node[midvertexR] (2) at (-0.782,-0.377) {};
					\node[midvertexR] (3) at (-0.782,1.623) {};
					\node[midvertexR] (4) at (0.782,-1.623) {};
					\node[midvertexR] (5) at (0.623,-0.218) {};
					\node[midvertexR] (6) at (0.218,0.623) {};
					\node[midvertexR] (7) at (1.623,0.782) {};
					\node[midvertexR] (8) at (0.782,0.377) {};
					\node[midvertexR] (9) at (-0.218,-0.623) {};
					\node[midvertexR] (10) at (-0.377,0.782) {};
					\node[midvertexR] (11) at (0.377,-0.782) {};
					\node[midvertexR] (12) at (-0.623,0.218) {};
					\draw[bedge] (1)edge(9);
					\draw[bedge] (1)edge(10);
					\draw[bedge] (2)edge(4);
					\draw[bedge] (2)edge(5);
					\draw[bedge] (3)edge(8);
					\draw[bedge] (3)edge(12);
					\draw[bedge] (4)edge(5);
					\draw[bedge] (6)edge(7);
					\draw[bedge] (6)edge(11);
					\draw[bedge] (7)edge(11);
					\draw[bedge] (8)edge(12);
					\draw[bedge] (9)edge(10);
					\draw[redge] (1)edge(11);
					\draw[redge] (1)edge(12);
					\draw[redge] (2)edge(3);
					\draw[redge] (2)edge(6);
					\draw[redge] (3)edge(6);
					\draw[redge] (4)edge(8);
					\draw[redge] (4)edge(9);
					\draw[redge] (5)edge(7);
					\draw[redge] (5)edge(10);
					\draw[redge] (7)edge(10);
					\draw[redge] (8)edge(9);
					\draw[redge] (11)edge(12);
				\end{scope}
			\end{tikzpicture}
	  \caption{A flexible $\C_4$-symmetric placement for a given $\C_4$-symmetric NAC-colouring.}\label{fig:flexrot-574214559646161505}
	\end{figure}
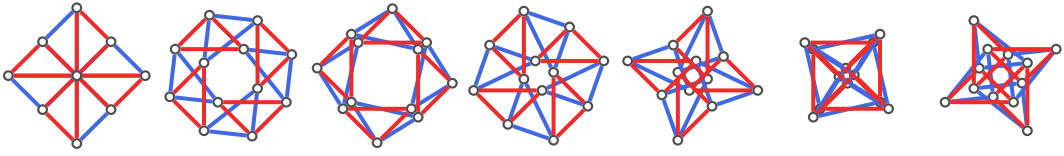
\end{example}

\section{\texorpdfstring{$\Cn$}{Cn}-symmetric flexible injective placements}
We call a framework \emph{proper} $\Cn$-symmetric flexible if it has a non-trivial flex
where all but finitely many $\Cn$-symmetric placements are injective.
The constructed framework given by the proof of Lemma~\ref{lem:NACImpliesflexible} may not be proper flexible.
A careful inspection shows that no two vertices coincide during the constructed flex
if and only if the conditions summarised in the following claim hold. 
\begin{proposition}
	Let a graph $G$ have a $\Cn$-symmetric NAC-colouring such that
	  $|V(B)\cap V(R)|\leq 1$ for each \blue\ component $B$ and \red\ component $R$,
	  no two \blue{}, resp.\ \red{}, partially invariant components are connected by a \red{}, resp.\ \blue{}, path, and
	  at most one vertex is in a \blue{} and \red{} partially invariant component simultaneously.
	Then $G$ has a proper flexible $\Cn$-symmetric placement.
\end{proposition}

We provide also a necessary combinatorial condition on 
the existence of a proper $\Cn$-symmetric flexible placement of a graph $G$ analogous to~\cite{movableGraphs}.
We call a NAC-colouring of $G$ \emph{active} for a $\Cn$-symmetric motion if it can be obtained from a valuation and threshold
using~\eqref{eq:valuation2NAC}.
We start with the following lemma.

\begin{lemma}
	\label{lem:constantDistance}
	Let $C$ be a $\Cn$-symmetric motion of $(G,p)$.
	Let $u,v$ be vertices of~$G$ such that $uv\not\in E(G)$
	and $q(u)\neq q(v)$ for all $q \in C$.
	If there exists a $uv$-path~$P$ in $G$ such that $P$ is monochromatic for all active NAC-colourings of $C$,
	then $||q(u)- q(v)||$ is the same for all $q \in C$. 
	Particularly, $C$ is a $\Cn$-symmetric motion of $\left(G',p\right)$,
	where $G'=(V(G),E(G)\cup \{\gamma u \gamma v : \gamma \in \Cn\})$.
\end{lemma}
\begin{proof}
	Using the fact that all active NAC-colourings are $\Cn$-symmetric (see the proof of Lemma~\ref{lem:flexibleImpliesNAC}),
	we can add the edge $uv$ to the graph $G$ by adapting the proof of~\cite[Lemma~3.1]{movableGraphs} 
	similarly to Lemma~\ref{lem:flexibleImpliesNAC} (if the angle between two consecutive edges $u_{i-1}u_i$, $u_iu_{i+1}$ in $P$ varies,
	i.e.,  the distance between $u$ and $v$ changes, then $\Wfun{u_i}{u_{i+1}}/\Wfun{u_{i-1}}{u_i}$ is a transcendental element).
	For $\gamma\in\Cn$, $\gamma u$ and $\gamma v$ are connected by the path $\gamma(P)$, which is monochromatic for all
	active NAC-colourings of $C$. Hence, we can use the same argument to add also $\gamma u\gamma v$. 
\end{proof}

We follow the approach in~\cite{movableGraphs}.
For a $\Cn$-symmetric graph $G$, let
$\upairsCn{G}$ denote all pairs $\{u,v\}\subset V(G)$ such that $uv\notin E(G)$ and there exists a path
from $u$ to~$v$ which is monochromatic for all $\Cn$-symmetric NAC-colourings of~$G$.	
Let $G_0, \dots, G_n$ be a sequence of $\Cn$-symmetric graphs such that
	$G=G_0$,
	$G_i=(V(G_{i-1}),E(G_{i-1}) \cup \upairsCn{G_{i-1}})$ for $i\in\{1,\dots,n\}$, and
	$\upairsCn{G_n}=\emptyset$.
The graph $G_n$ is called \emph{the $\Cn$-symmetric constant distance closure of $G$}, denoted by $\cdcCn{G}$.
In the same manner as in~\cite{movableGraphs}, we get the following statement.
\begin{theorem}
	A $\Cn$-symmetric graph $G$ admits a proper $\Cn$-symmetric flexible placement
	if and only if $\cdcCn{G}$ does as well. 
	In particular, if $\cdcCn{G}$ is complete, then $G$ has no proper $\Cn$-symmetric flexible placement.
\end{theorem}

While we have only dealt with frameworks with rotational symmetry,
there are also reflectional and translational symmetry in the plane.
Although flexible placements that preserve translational symmetry have very recently been investigated \cite{Dperiodic2019},
not much is known for flexible placements that preserve reflectional symmetry or preserve both reflectional and rotational symmetry.

\paragraph{Acknowledgments.}
This project has received funding from the European Union's Horizon~2020 research and innovation programme under the Marie Sk\l{}odowska-Curie grant agreement No~675789.
The project was supported by the Austrian Science Fund (FWF): P31061, P31888 and W1214-N15, 
and by the Ministry of Education, Youth and Sports of the Czech Republic, project no. CZ.02.1.01/0.0/ 0.0/16\_019/0000778.


\end{document}